\documentclass[12pt,a4paper]{amsart}
\usepackage[letterpaper, margin=0.8in]{geometry}
\usepackage{amsfonts,amssymb,amsthm,epsf, graphics, bbm, mathabx, verbatim, caption, subcaption, enumerate, amsmath}
\usepackage{comment}
\usepackage{hyperref}
\usepackage{diagbox}

\usepackage{indentfirst}

\newcommand{\ZZ}{\mathbb{Z}}
\newcommand{\RR}{\mathbb{R}}
\newcommand{\CC}{\mathbb{C}}
\newcommand{\NN}{\mathbb{N}}

\newcommand{\norm}[1]{\left\lVert#1\right\rVert}

\newcommand{\eps}{\varepsilon}
\newcommand*\conj[1]{\overline{#1}}

\newtheorem{theorem}{Theorem}[section]
\newtheorem{corollary}[theorem]{Corollary}
\newtheorem{proposition}[theorem]{Proposition}
\newtheorem{lemma}[theorem]{Lemma}

\newtheorem{conjecture}[theorem]{Conjecture}

\newtheorem{thm}{\bf Theorem}

\newtheorem{ppt}[thm]{Proposition}
\newtheorem{lem}[thm]{Lemma}

\theoremstyle{remark}

\theoremstyle{definition}

\usepackage{mathtools}
\usepackage{tikz}
\title{A quadratic Roth theorem for sets with large Hausdorff dimensions}
\author{Junjie Zhu}
\date{\today}

\keywords{
Fractals, polynomial configurations, Hausdorff dimension, Fourier transforms of measures
}
\subjclass[2020]{28A80, 42A38}

\begin{document}

\begin{abstract}
    Many results in harmonic analysis and geometric measure theory ensure the existence of geometric configurations under the largeness of sets, which are sometimes specified via the ball condition and Fourier decay. Recently, Kuca-Orponen-Sahlsten and Bruce-Pramanik proved Sarkozy-like theorems, which remove the Fourier decay condition and show that sets with large Hausdorff dimensions contain two-point patterns. This paper explores the existence of a three-point configuration that relies solely on the Hausdorff dimension. 
\end{abstract}

\maketitle

\section{Introduction}

\subsection{Pattern Recognition}

Finding the maximum size of a set that avoids certain patterns is an intriguing task with many applications. A well-known result \cite{Roth} by Roth shows that if $A \subset [N] = \{1, 2, \ldots, N\}$ and $A$ avoids any non-trivial 3-term arithmetic progression, then $|A| = O(N/\log\log N)$. An analog question in the continuum is also studied. If a set $A \subset [0, 1]$ avoids any arithmetic progression, it is Lebesgue null \cite{Jasinski}. Many works have examined the existence and avoidance of three-term arithmetic progressions of Lebesgue null sets. A set contains a three-term arithmetic progression if it supports a measure with a ball and a Fourier decay condition.

\begin{thm}\cite{LP}
\label{t_lp}
Let $C_1, C_2$, and $B$ be positive real numbers and $\beta \in (\frac{2}{3}, 1]$. Then, there exists $\eps_0>0$ such that the following statement holds: if $E \subset [0, 1]$ is a closed set and $\mu \in \mathcal{M}(E)$ has the following properties:
\begin{enumerate}
    \item (ball condition) $\mu([x, x+r]) \leq C_1 r^{\alpha}$ for an $\alpha \in (1 - \eps_0, 1)$ and all $r \in (0, 1]$, 
    \item (Fourier decay condition) $|\widehat{\mu}(k)| \leq C_2 (1-\alpha)^{-B} |k|^{-\frac{\beta}{2}}$ for all $k \neq 0$,
\end{enumerate}
then E contains a non-trivial 3-term arithmetic progression.
\end{thm}
In contrast, some Salem sets, which are traditionally considered large fractal sets, avoid any three-term arithmetic progression \cite{Shmer}.

A similar problem is finding the maximum size of the sets $A \subset [N]$ avoiding the pattern $\{x, x+z^2\}$ for some $x, z \in \NN$ as studied by \cite{Furstenberg, Sarkozy}. One variant of this pattern is the \textit{three-point quadratic pattern} $\{x, x+t, x+t^2\}$ in $\RR$. Any set $E \subset \RR$ with a positive Lebesgue measure contains $x, x+t, x+t^2$ for some $x, t \in \RR$ with $t\neq 0$. This is a special case of a general statement in the appendix (Lemma \ref{pos_leb_quad}). From \cite{Bourgain}, if $S \subset [0, N]$ with $|S| \geq \eps N$, then $S$ also contains $\{x, x+t, x+t^2\}$ for some $x, t \in \RR$ and $t$ greater than an absolute constant that depends on $N$ and $\eps$. One direction of study is to identify criteria for Lebesgue null sets to contain the same quadratic pattern. Fraser-Guo-Pramanik obtained the following result on the three-term polynomial pattern.

\begin{thm}\cite{fgp}
\label{t_fgp}
    Let $P: \RR \to \RR$ be a polynomial without a constant term and $\deg P \geq 2$. There exists $s_0>0$ such that the following statement holds. For $C_1, C_2$, $B$  positive real numbers, and $\beta \in (1-s_0, 1]$,  there exists $\eps_0>0$ such that if $E \subset [0, 1]$ is a closed set and $\mu \in \mathcal{M}(E)$ with the following properties:
\begin{enumerate}
    \item (ball condition) $\mu([x, x+r]) \leq C_1 r^{\alpha}$ for an $\alpha \in (1 - \eps_0, 1)$ and all $r \in (0, 1]$, 
    \item (Fourier decay condition) $|\widehat{\mu}(k)| \leq C_2 (1-\alpha)^{-B} |k|^{-\frac{\beta}{2}}$ for all $k \in \ZZ \backslash \{ 0\}$,
\end{enumerate}
then E contains three points $x, x+t, x+P(t)$ for real numbers $x, t$ with $t \neq 0$.
\end{thm} 

Another variant is the \textit{two-point quadratic pattern} $\{(x, y), (x+z, y+z^2)\}$ in $\RR^2$. Kuca-Orponen-Sahlsten studied the problem of identifying two-point quadratic patterns in $\RR^2$ for Lebesgue null sets using the Hausdorff dimension only.

\begin{thm}\cite{kos}
\label{t_kos}
    There is an absolute constant $\eps>0$ such that the following holds. Let $K \subset \RR^2$ with $\dim_H(K) > 2-\eps$. Then, there exists $x, y , z \in \RR$ with $z \neq 0$, such that
    $$\{(x, y), (x+z, y+z^2)\} \subset K.$$
\end{thm}

The three-point quadratic pattern in $\RR$ (Theorem \ref{t_fgp}) and the two-point quadratic pattern in $\RR^2$ (Theorem \ref{t_kos}) have different criteria to ensure their existence. We would like to understand their fundamental differences. Specifically, we investigate whether having a large Hausdorff dimension alone is sufficient to ensure the existence of the three-point quadratic pattern. Our conjecture is as follows.

\begin{conjecture}
    \label{c1}
    For all (Borel) $E \subset [0,1]$ with $\dim_H(E) > \frac{1}{2}$, there exist $x, t \in \RR$ with $t \neq 0$ such that
    $$\{x, x+t, x+t^2 \} \subset E.$$
\end{conjecture}

We require $\dim_H(E) > \frac{1}{2}$ because of the following result.

\begin{thm}\cite{mathe}
 There exists a compact set $E \subset \RR$ of Hausdorff dimension $ \frac{1}{2}$ such that $E$ does not contain $3$ distinct points $x, y, z$ 
satisfying $(z-x)=(y-x)^2$.
\end{thm} 

In this manuscript, we present results from adapting methods in \cite{kos} and \cite{BP} that partly address Conjecture \ref{c1}.

\subsection{Main Results}

The main theorem is as follows.
\begin{theorem}
\label{t_main}
     Given $q_0 \in [0, 1)$, there exists $\beta=\beta(q_0) \in (0, 1)$ with the following property. For every (Borel) $E \subset \RR$ with $\dim_H(E) > \beta$, one can find $0<a_0<a_1$ such that for every $(a, q)$ with $a_0 \leq |a| \leq a_1$, $|q| \leq q_0$,
     $$\{x, x+t, x+P_{a, q}(t) \} \subset E$$
     for some $x, t\in \RR$, $t \neq 0$. Here, $P_{a, q}(t) = at^2+qt$.
\end{theorem}
    As $q_0$ increases, $\beta$ increases. The proof of Theorem \ref{t_main} follows from the proposition below, which uses the \textit{Hausdorff content} defined as
\begin{equation}
    \mathcal{H}^{\beta}_{\infty}(S) := \inf \left\{\sum_{i} l(Q_i)^{\beta} : S \subset \bigcup_i Q_i, Q_i \subset [0, 1]\right\}
\end{equation}
for a set $S \subset [0, 1]$.

\begin{proposition}
\label{p_main}
    Let $p_0 < p_1$ be positive real numbers and $q_0 \in [0, 1)$.
     There exist $\beta \in (0, 1)$ and $\delta > 0$, such that the following statement holds: for all (Borel) $E \subset [0,1]$ with $\mathcal{H}_{\infty}^{\beta}(E) \geq 1-\delta$ and $P_{p, q}(t) = pt^2+qt$ for all $(p, q)$ with $p_0 \leq |p| \leq p_1$ and $|q| \leq q_0$,
    $$\{x, x+t, x+P_{p, q}(t) \} \subset E$$
    for some $x, t \in \RR$ with $t \neq 0$.
\end{proposition}


This is the first result confirming 3-point nonlinear patterns in sets of large Hausdorff dimensions alone. A key idea in the proof is that even though a set with a large Hausdorff dimension may not support a measure with any Fourier decay, a set with a large Hausdorff content supports a measure satisfying the spectral gap condition as stated in (\ref{sg}) below. This idea applies to many pattern recognition problems. For example, Bruce-Pramanik \cite{BP} adapted this idea to find two-point non-linear patterns for sets in $\RR^{d}$ ($d \geq 2$) with large Hausdorff dimensions.

We note that the proof strategy and many estimates are analogous to those in \cite{kos}. A key difference is that one remainder term in the proof of \cite{kos} uses the Fourier decay of measures supported on the parabola on $\RR^2$, an application of oscillatory integrals. In contrast, in this setting, the Sobolev improving estimate (Propositions \ref{cor_sie} and \ref{th_sie}) is the application of oscillatory integrals that utilize the non-vanishing Gaussian curvature of the curve $(t, P_{p, q}(t))$ in $\RR^{2}$.

\subsection{Application: Configuration set}
  An object to study in pattern recognition is the \textit{3-point configuration set} defined as $$\Delta_q(E) := \left\{\left. \frac{(z-x)-q(y-x)}{(y-x)^2} \right\rvert x,y,z \in E, x \neq y, x \neq z\right\}.$$
    If $E$ contains two points, $\Delta_q(E)$ is not empty. We remark that Conjecture \ref{c1} is equivalent to $1 \in \Delta_0(E)$
        if $\dim_H(E) > \frac{1}{2}$. The conjecture is also equivalent to
        $$(0, \infty) \subset \Delta_0(E)$$
        whenever $\dim_H(E)> \frac{1}{2}$ since the Hausdorff dimension is invariant under scaling. For example, let $a > 0$ and suppose that there exist $y, s \in \RR$ such that $\{y, y+s, y+s^2\} \subset aE$. By letting $x=a^{-1}y$ and $t = a^{-1}s$, we have $\{x, x+t, x+at^2\} \subset E$.

Greenleaf-Iosevich-Taylor \cite{GIT} studied 3-point configuration sets of general patterns for sets $E \subset \RR^{d}$, where $d \geq 2$. Our Theorem \ref{t_main} extends their result to sets in $\RR$. If $\{x, x+t, x+at^2+qt\} \subset E$ for some $x, t \in \RR$ with $t\neq 0$, by letting $y=x+t$ and $z=x+at^2+qt$, we note that $a \in \Delta_q(E)$. Therefore, a corollary of Theorem \ref{t_main} is as follows.
\begin{corollary}
     Given $q_0 \in [0, 1)$, there exists $\beta=\beta(q_0) \in (0, 1)$ with the following property. If $\dim_H(E) > \beta$ for a $\beta < 1$, there exist $0 < a_0 < a_1$, such that for all $q \in [-q_0, q_0]$,
    $$[-a_1, -a_0] \cup [a_0, a_1] \subset \Delta_q(E).$$
\end{corollary}
This corollary is a different generalization of \cite{GIT}, as it establishes the existence of the same interval in a family of configuration sets.

\subsection{Acknowledgments}
I thank Angel Cruz, Chun-Kit Lai, Yuveshen Moorogen, Malabika Pramanik, and Chenjian Wang for their discussions on this project and their guidance in preparing this manuscript. I thank Tuomas Orponen for correcting earlier versions of the manuscript and drawing my attention to the preprint \cite{krause}.

\subsection{Notations}

We denote $\mathcal{M}(S)$ as the set of Borel probability measures supported on $S$, and $\mathcal{L}$ as the Lebesgue measure.

For two functions $f, g: D \to \RR_{\geq 0}$ with a domain $D$, we write $f \lesssim g$ to denote that there exists a $c>0$, such that for all $x\in D$, $f(x) \leq c g(x)$.

\section{Proof of Theorem \ref{t_main} given Proposition \ref{p_main}}
\label{sec_2}

We denote the dyadic intervals as follows. For $j \geq 0$, let $\mathcal{D}_j$ be the set of dyadic intervals $[k2^{-j}, (k+1)2^{-j}]$ where $k \in \{0, 1, \ldots 2^{j}-1  \}$. Let $\mathcal{D} = \bigcup_{j \geq 0} \mathcal{D}_j$. For an interval $Q \in \mathcal{D}_j \subset \mathcal{D}$, its diameter $l(Q) = 2^{-j}$.  The associated re-scaling map $T_Q: Q \to [0, 1]$ is $T_Q(y) = 2^j(y-x_Q)$, where $x_Q$ is the left endpoint of $Q$. Then $T_Q(Q) = [0, 1]$. 

For a measure $\mu \in \mathcal{M}([0, 1])$, if $\mu(Q) > 0$, we can define $\mu_Q \in \mathcal{M}([0, 1])$ to be $\mu_Q = \mu(Q)^{-1} T_Q(\mu | Q)$, where $\mu|Q(S) = \mu(S \cap Q)$.

Given $p_0, p_1, q_0$ with $0< p_0 < p_1$ and $q_0 \in [0, \infty)$, we obtain $\beta$ and $\delta$ via Proposition \ref{p_main}. Now, suppose that $\dim_H(E) > \beta$. Let $H := \mathcal{H}^{\beta}_{\infty}(E) \leq 1$. 
We claim that there exists $\mathbf{Q} \in \mathcal{D}$ depending on $E$, such that 
\begin{equation}
\label{eq_bar}
    \mathcal{H}^{\beta}_{\infty}(E \cap \mathbf{Q}) \geq (1-\delta) l(\mathbf{Q})^{\beta}.
\end{equation}
Suppose, for a contradiction, that the claim is false. For a $\tau > 0$ sufficiently small such that $(1-\delta)(H + \tau) < H$, there exists $\{Q_j\}_{j \in \NN} \subset \mathcal{D}$, such that $\sum_{j} l(Q_j)^{\beta} \leq H + \tau$. Since (\ref{eq_bar}) is false for all $Q \in \mathcal{D}$,
$$H = \mathcal{H}^{\beta}_{\infty}(E) \leq \sum_{j} \mathcal{H}^{\beta}_{\infty} (E \cap Q_j) \leq (1-\delta) \sum_j l(Q_j)^{\beta} \leq (1-\delta)(H + \tau) < H.$$

Using $\mathbf{Q}$ from the claim above, we deduce that
$$\mathcal{H}^{\beta}_{\infty}(T_{\mathbf{Q}}(E\cap \mathbf{Q})) = l(\mathbf{Q})^{-\beta}\mathcal{H}^{\beta}_{\infty}(E\cap \mathbf{Q}) \geq 1-\delta.$$
By Proposition \ref{p_main}, when $p_0 \leq |p| \leq p_1$, and $|q| \leq q_0$, there exists $y, s$ with $s\neq 0$, such that
$$\{y, y+s, y+ps^2+qs\} \subset T_{\mathbf{Q}}(E\cap \mathbf{Q}).$$
Suppose that $l(\mathbf{Q}) = 2^{-J}$ for a $J \geq 0$. Then,
$$\{2^{-J}y+x_{\mathbf{Q}}, 2^{-J}(y+s)+x_{\mathbf{Q}}, 2^{-J}(y+ps^2+qs)+x_{\mathbf{Q}}\} \subset E \cap \mathbf{Q} \subset E.$$
Finally, we perform a substitution $x = 2^{-J}y+x_{\mathbf{Q}}, t = 2^{-J}s$ such that there exist $x, t$ with $t \neq 0$, such that
$$\{x, x+t, x+ 2^{J}p t^2+ qt\} \subset E,$$
and the conclusion of Theorem \ref{t_main} holds if we let $a_0 = 2^{J} p_0$, $a_1=2^{J} p_1$.

\section{Proof of Proposition \ref{p_main}}
\label{sec_3}
  
Let $l \in \NN$, $\tau \in \mathcal{S}(\RR)$, where $\tau(x) = 1$ if $x \in [1, 2]$, $\text{spt } \tau \subset [2^{-1}, 2^{2}]$, and $\tau_l(x) = \tau(2^lx)$. We apply the following proposition, which allows us to examine pattern recognition problems with analytical tools.  Let 
\begin{equation}
\label{phi_choice}
    \phi \in \mathcal{S}(\RR) \text{ be even, real, non-negative, } \widehat{\phi}(0) = \int \phi(x)dx = 1, \phi(x) \geq 2^{-1} \text{ if } |x| \leq 2^{-1},
\end{equation} $ \text{spt } \phi \subset [-1, 1],$ and $\phi_{\epsilon}(x) = \epsilon^{-1} \phi(\epsilon^{-1} x)$. 
\begin{ppt}\cite[Section 6]{fgp}
\label{p_pos}
    Suppose
    \begin{equation}
        \label{eq_lim_pos}
        \liminf_{\eps \to 0}\int \int \mu_{\eps}(x+t) \mu_{\eps}(x+P_{p, q}(t)) \tau_l(t) dt d\mu(x) > 0,
    \end{equation}
    where $\mu \in \mathcal{M}(E)$, $\mu_{\epsilon} = \mu * \phi_{\epsilon}$, and $I_{1-\gamma}(\mu)$ is finite.
    Then, there exist $x \in E$ and $t \in [2^{-l-1}, 2^{-l+2}]$, such that $\{x, x+t, x+P_{p, q}(t) \} \subset E$.
\end{ppt} 

The integral in (\ref{eq_lim_pos}) is often called the configuration integral. A technical tool used to bound the configuration integral is the Sobolev improving estimate stated below. We use the Sobolev norm defined in (\ref{sob_norm}).
\begin{proposition}\cite{fgp}
\label{cor_sie}
There exist $\gamma_0>0$ and $\kappa >0$ such that the following statement holds.
Let $\gamma \in (0, \gamma_0)$, $p_0 > 0$, and $q_0 \geq 0$. Then, there exists $C_{\gamma, p_0, q_0} > 0$ such that
\begin{equation*}
\left|\int \int f(x+t) g(x+P_{p, q}(t)) \tau_l(t)dt  d\mu(x) \right| \leq C_{\gamma, p_0, q_0} 2^{\kappa l} \norm{f}_{H^{-\gamma}}  \norm{g}_{H^{-\gamma}} \norm{\mu}_{H^{-\gamma}}
\end{equation*}
for all $f, g \in \mathcal{S}(\RR)$ if $P_{p, q}(t) = pt^2+qt$ for $|p| \geq p_0$ and $|q| \leq q_0$.
\end{proposition}

The constant $C_{\gamma, p_0, q_0}$ increases if $\gamma$ increases, $p_0$ decreases, or $q_0$ increases. We note that $\tau_l$ in the integrand of (\ref{eq_lim_pos}) is used to enable the application of the Sobolev improving estimate (Proposition \ref{cor_sie}), and the integral in (\ref{eq_lim_pos}) is finite for each $\eps > 0$ by applying Proposition \ref{cor_sie}. We refer readers to Section \ref{sec_sie} for more discussion.

The measure we use satisfies special properties that hold when the Hausdorff content is large. The \textit{s-energy} $I_s$ is provided in (\ref{s_energy}).
\begin{lem}\cite{kos}
\label{p_sg}
Given $0 < A < B$, there exist $\delta = \delta(A, B)$ and $\beta=\beta(A, B)$, such that the following statements hold: for $E \subset [0, 1]$, suppose
$$\mathcal{H}_{\infty}^{\beta}(E) \geq 1-\delta.$$
Then, there exists a measure $\mu \in \mathcal{M}(E)$, such that
    \begin{enumerate}[a.]
         \item If $s < \beta$, $$I_{s}(\mu) \leq 1 + \frac{16Cs}{\beta-s},$$
         where $C$ is an absolute constant from Frostman's lemma.
         \item $\mu$ satisfies the spectral gap condition that 
\begin{equation}
\label{sg}
\int_{A^{\frac{1}{5}} \leq |\xi| \leq B^2} |\widehat{\mu}(\xi)|^2 d\xi \leq \int_{A^{\frac{1}{5}} \leq |\xi| \leq B^2} |\widehat{\mu}(\xi)| d\xi \leq A^{-3}.   
\end{equation}
    \end{enumerate}
\end{lem}

Lemma \ref{p_sg} is a one-dimensional variant of section 3 in \cite{kos}. A proof of Lemma \ref{p_sg} is also given in section \ref{sec_sg}. Hence, Proposition \ref{p_main} follows from the lemma below.

\begin{lemma}
\label{p_ci_bound} 
Given $p_0, p_1, q_0 \in \RR_{+}$ with $q_0 < 1$, there exists $A > 0$, $l = \log_2 A+3$, $B$, $\delta$, $\beta$, such that when $$\mathcal{H}_{\infty}^{\beta}(E) \geq 1-\delta,$$ for $E \subset [0, 1]$,
there exists a measure $\mu \in \mathcal{M}(E)$ from Lemma \ref{p_sg}, such that for $\eps > 0 $ sufficiently small,
    \begin{equation}
    \label{eq_ci}
  \int \int \mu_{\eps}(x+t) \mu_{\eps}(x+P_{p, q}(t)) \tau_l(t) dt d\mu(x) \gtrsim A^{-1}.
\end{equation}
\end{lemma}

\subsection{Positivity of the configuration integral}

The strategy is to apply 
$$\mu_{\epsilon} = \mu_{A^{-1}}+(\mu_{B^{-1}} - \mu_{A^{-1}})+(\mu_{\epsilon}-\mu_{B^{-1}}) := \mu_{A^{-1}} + \mu_{\text{mid}}+\mu_{\text{high}}$$
to rewrite the first two factors of the integrand in (\ref{eq_ci}). Then, (\ref{eq_ci}) is decomposed into 9 terms. The main term of (\ref{eq_ci}) is
\begin{equation*}
     \int \int \mu_{A^{-1}}(x+t) \mu_{A^{-1}}(x+P_{p, q}(t)) \tau_l(t) dt d\mu(x). 
\end{equation*}

The rest of the terms are remainder terms. We call a remainder term Type I if the function $(\mu_{\eps}-\mu_{B^{-1}})=\mu_{\text{high}}$ is not part of the integrand. Otherwise, the remainder term is categorized as a Type II remainder term. There are $3$ Type I remainder terms and $5$ Type II remainder terms. 

 Table \ref{table_1} records the estimates of all 9 terms. Except for the main term, all bounds are upper bounds. Additionally, unless stated otherwise all remainder terms are Type I.
    
    \begin{table}[!ht]
    \centering
    \renewcommand{\arraystretch}{2}
    \begin{tabular}{|l|*{3}{c|}}\hline
& $\mu_{A^{-1}}(x+P_{p, q}(t))$ & $\mu_{\text{mid}}(x+P_{p, q}(t))$ & $\mu_{\text{high}}(x+P_{p, q}(t))$ 
\\\hline\hline
$\mu_{A^{-1}}(x+t)$ & Main: $\gtrsim A^{-1}$ (\ref{e_main}) & Type I: $ 2^{-l}  A^{-\frac{2}{5}}$ (\ref{eqab}) & $ \mathbf{C}^{\frac{3}{2}}   2^{\kappa l}  B^{\frac{1-s-\gamma}{10}}$ (\ref{eqae})\\\hline
$\mu_{\text{mid}}(x+t)$ & Type I: $ 2^{-l} A^{-\frac{2}{5}}$  & Type I: $ 2^{-l} A^{-\frac{14}{5}}$  & $\mathbf{C} 2^{\kappa l}   A^{-\frac{3}{2}} B^{\frac{1-s-\gamma}{10}}$ \\\hline
$\mu_{\text{high}}(x+t)$ &  $  \mathbf{C}^{\frac{3}{2}}  2^{\kappa l} B^{\frac{1-s-\gamma}{10}}$ & $ \mathbf{C} 2^{\kappa l} A^{-\frac{3}{2}} B^{\frac{1-s-\gamma}{10}}$ & $ \mathbf{C}^{\frac{3}{2}}  2^{\kappa l}   B^{\frac{1-s-\gamma}{5}} $  \\\hline
\end{tabular}
    \caption{Bounds on all terms assuming $I_s(\mu), I_{1-\gamma}(\mu) \leq \mathbf{C}$ for a $\mathbf{C}>0$.} 
    \label{table_1}
\end{table}

The lower bound on the main term (\ref{e_main}) is shown in Lemma \ref{lem31}, and an overview of the upper bounds of all remainder terms is provided in Section \ref{sec_e_rem}, where Lemmas \ref{lemt1} and \ref{lemt2} contain detailed proofs of the bounds (\ref{eqab}) and (\ref{eqae}), respectively.

\begin{proof}[Proof of Lemma \ref{p_ci_bound}]
    We use $\gamma$ from Proposition \ref{cor_sie}. First, we choose two positive real numbers $s < \beta'$ that can be expressed as $s = 1- \alpha_1 \gamma$, $\beta' = 1- \alpha_2 \gamma$ for $1 > \alpha_1 > \alpha_2 > 0$, where $\gamma$ comes from the Sobolev improving estimate in Proposition \ref{cor_sie}.
    Let $\textbf{C} = 1+\frac{16 C s}{\beta'-s}$.
    
    We choose $l$, $A=2^{l-3}$, such that the lower bound of the main term (\ref{e_main}) is large and the upper bounds on Type I remainder terms are small. Finally, we choose $B$ so that the Type II remainder terms are small.
We then apply Proposition \ref{p_sg} with the chosen $A$, $B$ to obtain $\beta$ and $\delta$. If $E \subset [0, 1]$ satisfies 
$$\mathcal{H}^{\max\{\beta, \beta'\}}_{\infty}(E) \geq 1-\delta,$$
there is a measure $\mu \in \mathcal{M}(E)$ with $I_s(\mu), I_{1-\gamma}(\mu) \leq \textbf{C}$ that satisfies (\ref{sg}). We then use the estimates in Table \ref{table_1} to bound the main term and each of the 8 remaining terms. With the appropriate choice of $l$, $A$, and $B$, Lemma \ref{lem31} applies, and the norm of each remainder term is smaller than a multiple of $A^{-1}$.
\end{proof}


\subsection{The main term}
\begin{lemma}
\label{lem31}
If $(4A)^{-1} = 2^{1-l}$ for an $l \in \NN$ and $A$ is sufficiently large such that $|P_{p, q}(t)| \leq |t|$ when $|t|\leq (4A)^{-1}$,
    \begin{equation}
    \label{e_main}
     \int \int \mu_{A^{-1}}(x+t) \mu_{A^{-1}}(x+P_{p, q}(t)) \tau_l(t) dt d\mu(x) \geq 2^{-10} c^2 A^{-1}. 
\end{equation}
for a $c> 0$.
\end{lemma}

\begin{proof}

The key to obtain (\ref{e_main}) is the following estimate. For $c>0$, let 
\begin{equation}
    D_c := \{x| \exists r_x \in (0, 1], \mu(B(x, r_x)) \leq cr_x \}.
\end{equation}
    We claim that there exists $c>0$, with
    $\mu(D_c) \leq 2^{-1}.$ Note that $$D_c \subset \bigcup_{x \in D_c}B(x, r_x).$$ By applying the Vitali covering lemma \cite[Theorem 1.24]{eg} to $\bigcup_{x \in D_c}B(x, r_x/5)$, there exist $\{x_j  \}_{j \in \NN} \subset D_c$ and $\{r_j\}_{j \in \NN} \subset \RR_{>0}$, such that
    $$D_c \subset \bigcup_{j=1}^{\infty}B(x_j, r_j),$$
    and $\mu(B(x_j, r_j)) \leq c r_j$.
    In addition, $B(x_j, r_j/5) \cap B(x_k, r_k/5) = \emptyset$ if $j \neq k$, but $\bigcup_{j=1}^{\infty} B(x_j, r_j/5) \subset [-2, 2]$. Then,
    \begin{equation*}
        \mu(D_c) \leq \sum_{j=1}^{\infty} \mu(B(x_j,r_j)) \leq \sum_{j=1}^{\infty} cr_j  =\frac{5c}{2}\sum_{j=1}^{\infty} \mathcal{L}(B(x_j, r_j/5)) \leq \frac{5c}{2} \mathcal{L}([-2, 2]) = 10c.
    \end{equation*}
    The claim holds when $c=20^{-1}$.

    Suppose that $x$ satisfies the condition that for all $r \in (0, 1)$, $\mu(B(x, r)) \geq cr$. Then, for $|t| \leq (4A)^{-1}$,
    \begin{equation*}
        \begin{aligned}
            \mu_{A^{-1}}(x+t) & = A \int \phi(A(x+t-y)) d\mu(y) \\
            & \geq \frac{A}{2} \mu(B(x+t, (2A)^{-1})) & \text{ from (\ref{phi_choice}), } \phi(x) \geq 2^{-1} \text{ if } |x| \leq 2^{-1}\\
            & \geq \frac{A}{2} \mu(B(x, (4A)^{-1})) \\
            & \geq \frac{A}{2} c(4A)^{-1} = \frac{c}{8} & \text{ since } x \in (D_c)^c.
        \end{aligned}
    \end{equation*}
    Similarly, since $|P_{p, q}(t)| \leq |t| \leq (4A)^{-1}$, we also have $\mu_{A^{-1}}(x+P_{p, q}(t)) \geq \frac{c}{8}$. To complete the proof,
        \begin{align*}
            & \int \int \mu_{A^{-1}}(x+t) \mu_{A^{-1}}(x+P_{p, q}(t)) \tau_l(t) dt d\mu(x) \\
            \geq & \int_{(D_c)^c} \int_{t \in (0, (4A)^{-1})} \mu_{A^{-1}}(x+t) \mu_{A^{-1}}(x+P_{p, q}(t)) \tau_l(t) dt d\mu(x) \\
            \geq & 2^{-1} \left(2^{-3} c\right)^2\mathcal{L}((0, (4A)^{-1}) \cap (2^{-l}, 2^{-l+1})) \\
            = & 2^{-10}c^{2} A.        \qedhere
        \end{align*}  
\end{proof}
\subsection{Estimate of remainder terms}
\label{sec_e_rem}

In this section, we present the estimates of the Type I and Type II remainder terms.  We first record some basic estimates that will prove useful in what follows. From our choice of $\phi$ in (\ref{phi_choice}), since $\frac{\partial \widehat{\phi}}{\partial \xi}(0) = - \int 2 \pi i x \phi(x)dx=0$, 
\begin{equation}
\label{lem_b_small_a}
    |\widehat{\phi}(\xi) - \widehat{\phi}(0)| \lesssim |\xi|^2,  \sup_{\xi \in \RR}|\xi|^5|\widehat{\phi}(\xi)| < \infty.
\end{equation}

First, since $\mu$ is a probability measure,
\begin{equation}
\label{lem_a}
    \int  |\widehat{\mu}(\xi)||\widehat{\phi}(A^{-1}\xi)| d\xi \lesssim A.
\end{equation}
Next, from the definition of the Sobolev norm in (\ref{sob_norm}) and the $s$-energy integral (\ref{s_energy}),
\begin{equation}
\label{lem_a_h}
    \norm{\mu_{A^{-1}}}_{H^{-\gamma}}^2,  \norm{\mu}_{H^{-\gamma}}^2 \lesssim  I_{1-\gamma}(\mu).
\end{equation}

Here are estimates obtained via the spectral gap condition.
\begin{lemma}
\label{lem_b}
Assuming the spectral gap condition (\ref{sg}), if $A$ is sufficiently large, 
 \begin{equation}
    \label{lem_b_e1}
    \begin{split}
         \int |\widehat{\mu}(\xi)||\widehat{\phi}(B^{-1}\xi)-\widehat{\phi}(A^{-1}\xi)| d\xi & \lesssim A^{-\frac{7}{5}}, \\
         \norm{\mu_{B^{-1}}-\mu_{A^{-1}}}_{H^{-\gamma}}^2 & \lesssim  A^{-3}.
    \end{split}
\end{equation}
\end{lemma}

\begin{proof}
    To show the first equation of (\ref{lem_b_e1}), we decompose the integral into three parts:
\begin{equation*}
    \begin{split}
        \int |\widehat{\mu}(\xi)||\widehat{\phi}(B^{-1}\xi)-\widehat{\phi}(A^{-1}\xi)| d\xi
        & = \int_{|\xi| \leq A^{\frac{1}{5}}} + \int_{ A^{\frac{1}{5}} \leq |\xi| \leq B^2} + \int_{|\xi| \geq B^2}.
    \end{split}
\end{equation*}
For the first integral, using (\ref{lem_b_small_a}) and $\norm{\widehat{\mu}}_{L^{\infty}} = 1$,
\begin{equation*}
    \begin{split}
       & \int_{|\xi| \leq A^{\frac{1}{5}}}  |\widehat{\mu}(\xi)||\widehat{\phi}(B^{-1}\xi)-\widehat{\phi}(A^{-1}\xi)| d\xi \\
       \leq  & \int_{|\xi| \leq A^{\frac{1}{5}}}  |\widehat{\phi}(B^{-1}\xi)-\widehat{\phi}(A^{-1}\xi)| d\xi 
       \lesssim  A^{-\frac{7}{5}}.
    \end{split}
\end{equation*}
For the second integral, by the spectral gap condition (\ref{sg}),
\begin{equation*}
    \begin{split}
       & \int_{ A^{\frac{1}{5}} \leq |\xi| \leq B^2}  |\widehat{\mu}(\xi)||\widehat{\phi}(B^{-1}\xi)-\widehat{\phi}(A^{-1}\xi)| d\xi \\
       \leq  & 2 \int_{ A^{\frac{1}{5}} \leq |\xi| \leq B^2} |\widehat{\mu}(\xi)|  d\xi 
       \leq 2 A^{-3}.
    \end{split}
\end{equation*}
For the third integral, by (\ref{lem_b_small_a}),
\begin{equation*}
    \begin{split}
       & \int_{|\xi| \geq B^2}  |\widehat{\mu}(\xi)||\widehat{\phi}(B^{-1}\xi)-\widehat{\phi}(A^{-1}\xi)| d\xi \\
       \leq  & \int_{|\xi| \geq B^2}  |\widehat{\phi}(B^{-1}\xi)-\widehat{\phi}(A^{-1}\xi)| d\xi 
       \lesssim  B^{-3}.
    \end{split}
\end{equation*}
As $A \leq B$, we combine the above to obtain the inequality required for $A \geq 1$. The proof of the second equation of (\ref{lem_b_e1}) is similar.
\end{proof}

The final estimate we need is as follows. 
\begin{lemma}
\label{lem_eps}
For $B$ sufficiently large, if $s >1-\gamma$,
\begin{equation}
\label{e_eps_b}
    \norm{\mu_{\eps}-\mu_{B^{-1}}}_{H^{-\gamma}}^2 \lesssim B^{\frac{1-s-\gamma}{5}} I_s(\mu).
\end{equation}
\end{lemma}

\begin{proof}
    To show (\ref{e_eps_b}), we decompose $\norm{\mu_{\eps}-\mu_{B^{-1}}}_{H^{-\gamma}}^2$ into two parts:
\begin{equation*}
    \begin{split}
        \norm{\mu_{\eps}-\mu_{B^{-1}}}_{H^{-\gamma}}^2 & = \int |\widehat{\mu_{\eps}}(\xi)-\widehat{\mu_{B^{-1}}}(\xi)|^2 (1+ |\xi|^2)^{-\frac{\gamma}{2}}d\xi \\
        & = \int |\widehat{\mu}(\xi)|^2|\widehat{\phi}(\eps\xi)-\widehat{\phi}(B^{-1}\xi)|^2 (1+ |\xi|^2)^{-\frac{\gamma}{2}}d\xi 
         = \int_{|\xi| \leq B^{\frac{1}{5}}}  + \int_{|\xi| \geq B^{\frac{1}{5}}}.
    \end{split}
\end{equation*}
For the first integral, using (\ref{lem_b_small_a}) and $\norm{\widehat{\mu}}_{L^{\infty}} = 1$,
\begin{equation*}
    \begin{split}
       & \int_{|\xi| \leq B^{\frac{1}{5}}}  |\widehat{\mu}(\xi)|^2|\widehat{\phi}(\eps\xi)-\widehat{\phi}(B^{-1}\xi)|^2 (1+ |\xi|^2)^{-\frac{\gamma}{2}}d\xi \\
       \leq  & \int_{|\xi| \leq B^{\frac{1}{5}}}  |\widehat{\phi}(\eps\xi)-\widehat{\phi}(B^{-1}\xi)|^2 d\xi 
       \lesssim B^{-3}.
    \end{split}
\end{equation*}
For the second integral, if $1-s-\gamma < 0$,
\begin{equation*}
    \begin{split}
       & \int_{|\xi| \geq B^{\frac{1}{5}}}  |\widehat{\mu}(\xi)|^2|\widehat{\phi}(\eps\xi)-\widehat{\phi}(B^{-1}\xi)|^2 (1+ |\xi|^2)^{-\frac{\gamma}{2}}d\xi \\
       \leq  & 4 \int_{|\xi| \geq B^{\frac{1}{5}}}  |\widehat{\mu}(\xi)|^2 (1+ |\xi|^2)^{-\frac{\gamma}{2}} d\xi \\
        \leq & 4 B^{\frac{1-s-\gamma}{5}}\int_{|\xi| \geq B^{\frac{1}{5}}}  |\widehat{\mu}(\xi)|^2 |\xi|^{s-1} d\xi = 4 B^{\frac{1-s-\gamma}{5}} I_s(\mu).
    \end{split}
\end{equation*}
Then, for $B$ sufficiently large, $\norm{\mu_{\eps}-\mu_{B^{-1}}}_{H^{-\gamma}}^2 \lesssim  B^{\frac{1-s-\gamma}{5}} I_s(\mu).$
\end{proof}

Therefore, we can deduce the following Type I remainder term estimate.
\begin{lemma}
\label{lemt1}
Under the spectral gap condition (\ref{sg}),
 \begin{equation}
\label{eqab}
     \left|\int \int \mu_{A^{-1}}(x+t) \mu_{\text{mid}}(x+P_{p, q}(t)) \tau_l(t) dt d\mu(x) \right| \lesssim 2^{-l}  A^{-\frac{2}{5}}.
\end{equation}
\end{lemma}

\begin{proof}[Proof of Lemma \ref{lemt1}]
By applying the Fourier transform (or (\ref{tof})), we have
    \begin{equation*}
    \begin{split}
        & \left| \int \mu_{A^{-1}}(x+t)  \mu_{\text{mid}}(x+P_{p, q}(t)) \tau_l(t) dt d\mu(x) \right| \\
         = & 2^{-l}\left| \int \int \widehat{\mu}(\xi+\eta)\conj{\widehat{\mu}}(\xi) \conj{\widehat{\phi}}(A^{-1}\xi)\conj{\widehat{\mu}}(\eta) (\conj{\widehat{\phi}}(B^{-1}\eta)-\conj{\widehat{\phi}}(A^{-1}\eta)) \int e^{-2\pi i (2^{-l}t\xi+P_{p, q}(2^{-l}t)\eta)} \tau(t)dtd\xi d\eta \right| \\
         \leq & 2^{-l} \norm{\tau}_{L^{1}}\int  |\widehat{\mu}(\xi)||\widehat{\phi}(A^{-1}\xi)| d\xi \int |\widehat{\mu}(\eta)||\widehat{\phi}(B^{-1}\eta)-\widehat{\phi}(A^{-1}\eta)| d\eta \\
         \lesssim & 2^{-l}  A^{-\frac{2}{5}} 
    \end{split}
    \end{equation*}
    by (\ref{lem_a}) and Lemma \ref{lem_b}.
\end{proof}

The estimates of the other error Type I remainder terms are similar. All Type II remainder terms' estimates are obtained by applying Proposition \ref{cor_sie}. Here is one example.

\begin{lemma}
\label{lemt2}
    Assuming the spectral gap condition (\ref{sg}) and $I_s(\mu), I_{1-\gamma}(\mu) \leq \textbf{C}$,
    \begin{equation}
\label{eqae}
\begin{split}
      \left|\int \int \mu_{A^{-1}}(x+t) (\mu_{\eps}-\mu_{B^{-1}})(x+t^2) \tau_l(t) dt d\mu(x) \right| \lesssim  \mathbf{C}^{\frac{3}{2}} 2^{\kappa l}  B^{\frac{1-s-\gamma}{10}} .
\end{split}
\end{equation}
\end{lemma}

\begin{proof}
    Using Proposition \ref{cor_sie},
    \begin{equation*}
        \begin{split}
    & \left|\int \int \mu_{A^{-1}}(x+t) (\mu_{\eps}-\mu_{B^{-1}})(x+t^2) \tau_l(t) dt d\mu(x) \right| \\
    \leq &  C_{\gamma, p_0, q_0} 2^{\kappa l} \norm{\mu_{A^{-1}}}_{H^{-\gamma}}  \norm{\mu_{\eps}-\mu_{B^{-1}}}_{H^{-\gamma}} \norm{\mu}_{H^{-\gamma}} \\
    \lesssim & 2^{\kappa l} [I_{1-\gamma}(\mu)]^{\frac{1}{2}} [B^{\frac{1-s-\gamma}{5}} I_s(\mu)]^{\frac{1}{2}} [I_{1-\gamma}(\mu)]^{\frac{1}{2}} 
        \end{split}
    \end{equation*}
    by (\ref{lem_a_h}) and Lemma \ref{lem_eps}.
\end{proof}

\section{Appendix}

The following identity is used to study the integral in the frequency space.
\begin{equation}
\label{tof}
\begin{split}
   & \int \int f(x+t) g(x+t^2) h(x)\tau_l(t)dtdx \\
   = & \int \int \widehat{h}(\xi+\eta)\conj{\widehat{f}}(\xi)\conj{\widehat{g}}(\eta) \int e^{-2\pi i (t\xi+t^2\eta)} \tau_l(t)dtd\xi d\eta \\
    = & 2^{-l}\int \int \widehat{h}(\xi+\eta)\conj{\widehat{f}}(\xi)\conj{\widehat{g}}(\eta) \int e^{-2\pi i (2^{-l}t\xi+2^{-2l}t^2\eta)} \tau(t)dtd\xi d\eta.
\end{split}
\end{equation}

The $s$-energy integral of a measure is given by
\begin{equation}
\label{s_energy}
    I_s(\mu) = \int \int |x-y|^{-s} d\mu(x) d\mu(y) = \rho_s \int |\widehat{\mu}(\xi)|^2|\xi|^{-\frac{s}{2}} d\xi,
\end{equation}
where
\begin{equation*}
\label{rho_gamma}
    \rho_s = \pi^{s-1/2}\frac{\Gamma(\frac{1-s}{2})}{\Gamma(\frac{s}{2})},
\end{equation*}
and $\Gamma$ is the Gamma function. We refer readers to \cite{mattila_2015} for details on the $s$-energy integral and its relation to the Hausdorff dimension.

The Sobolev norm is given by
\begin{equation}
\label{sob_norm}
    \norm{f}_{H^{s}}^2 := \int |\widehat{f}(\xi)|^2 (1+|\xi|^2)^{\frac{s}{2}} d\xi
\end{equation}
for $f \in \mathcal{S}(\RR)$ and $s \in \RR$.

\subsection{Positive Lebesgue measure sets}

In the introduction, we state that every set with a positive Lebesgue measure contains the pattern $\{x, x+t, x+t^2\}$. We prove a more general statement below.

\begin{lemma}
\label{pos_leb_quad}
    Let $f_j: \RR^{d} \to \RR^d$ be continuous functions with $f_j(0)=0$ for $j \in \{1, \ldots, n\}$. If $E \subset \RR^{d}$ and the Lebesgue measure $|E| > 0$, then there exist $x, t \in \RR^d$, $t \neq 0$, such that $\{x+f_j(t)\}_{j=1}^{n} \subset E$.
\end{lemma}

\begin{proof}
    Let
    \begin{equation*}
        I(t) = \left|\bigcap_{j=1}^{n} (E-f_j(t)) \right|.
    \end{equation*}
    Then, for $t \in \RR^d$, $I(t) \leq I(0) = |E|>0$. In addition,
    \begin{equation*}
        I(0)-I(t) \leq |E\backslash \bigcap_{j=1}^{n} (E-f_j(t)) | \leq \sum_{j=1}^{n} |E \backslash (E-f_j(t))|.
    \end{equation*}
    We note that by the $L^1$-modulus of continuity, $$\int_{\RR^{d}} |f(x+h)-f(x)| dx \to 0 \text{ as } h \to 0$$
    for $f \in L^{1}(\RR^{d})$. Therefore, by applying the above with $f=\mathbf{1}_{E}$, we have $|E\backslash(E-f_j(t))| \to 0$ as $t \to 0$.
\end{proof}

\subsection{Existence of a measure with a spectral gap}
\label{sec_sg}

We present the one-dimensional version of section 3 in \cite{kos}, where dyadic rectangles of size $2^{-T}\times4^{-T}$ are examined for a $T \geq 1$. For $\RR$, we use dyadic intervals of length $2^{-T}$.

\begin{proof}[Proof of Lemma \ref{p_sg}]

Here, we use the dyadic intervals as stated in the proof of Theorem \ref{t_main} in section \ref{sec_2}. Let 
\begin{equation}
\label{varphi_choice}
    \varphi \in C^{\infty}(\RR), \text{spt }\varphi \subset (0, 1), \int \varphi(x) dx = 1, \text{ and } \norm{\varphi}_{L^{\infty}} = 2.
\end{equation}
For $N \in \NN$ large, there exists $C_N$, such that $|\xi|^N|\widehat{\varphi}(\xi)| \leq C_N$. Hence,
\begin{equation*}
    \begin{split}
        \int_{|\xi| \geq A^{\frac{1}{5}}} |\widehat{\varphi}(\xi)| d\xi \leq  C_N \int_{|\xi| \geq A^{\frac{1}{5}}} |\xi|^{-N} d\xi = \frac{2C_N}{N-1}A^{\frac{1-N}{5}}.
    \end{split}
\end{equation*}
If $N = 21$, $\frac{C_{21}}{5} \leq A$, then $\int_{|\xi| \geq A^{\frac{1}{5}}} |\widehat{\varphi}(\xi)| d\xi \leq 2^{-1}A^{-3}$. 

Let $T \in \NN$ such that $2^{-T+3}B^4 \leq 2^{-1}A^{-3}$. Suppose that $\mathcal{H}^{\beta}_{\infty}(E) \geq 1-\delta$, where $\delta = 2^{-3T-3}$. We denote $\text{ch}(Q) \subset \mathcal{D}$ to be the generation-$T$ children of $\mathbf{Q} = [0, 1]$.
We claim that if $1-\beta$ is small, 
$\mathcal{H}^{\beta}_{\infty}(E \cap Q) \geq 2^{-1} l(Q)^{\beta}$ for $Q \in \text{ch}(\mathbf{Q})$.

\begin{itemize}
    \item To see that the claim holds, let $$\mathcal{G} = \{ Q \in \text{ch}(\mathbf{Q}): \mathcal{H}^{\beta}_{\infty}(E \cap Q ) \geq 2^{-1} l(Q)^{\beta}  \}.$$ If $\mathcal{G} \subsetneq \text{ch}(\mathbf{Q})$, 

\begin{equation*}
    \begin{aligned}
        1-2^{-3T-3} = 1-\delta \leq & \mathcal{H}^{\beta}_{\infty}(E \cap \mathbf{Q}) \\
        \leq & \sum_{Q \in \mathcal{G}} \mathcal{H}^{\beta}_{\infty}(E \cap Q) + \sum_{Q \in \text{ch}(\mathbf{Q}) \backslash \mathcal{G}} \mathcal{H}^{\beta}_{\infty}(E \cap Q) \\
        \leq &  \sum_{Q \in \mathcal{G}} l(Q)^{\beta} + (1-2^{-1})\sum_{Q \in \text{ch}(\mathbf{Q}) \backslash \mathcal{G}} l(Q)^{\beta} \\
        \leq & \sum_{Q \in \text{ch}(\mathbf{Q})} l(Q)^{\beta} - 2^{-1} 2^{-T\beta} & \text{ since }|\text{ch}(\mathbf{Q}) \backslash \mathcal{G}| \geq 1\\
        =  & 2^{-T\beta}(2^T-2^{-1}).
    \end{aligned}
\end{equation*}

However, for $\beta$ sufficiently close to $1$, $1-2^{-3T-3} > 2^{-T\beta}(2^T-2^{-1})$, which contradicts the inequality above.
\end{itemize}

We apply Frostman's lemma \cite[Theorem 2.7]{mattila_2015} to $E \cap Q$ for each $Q \in \text{ch}(\mathbf{Q})$ to obtain a Borel measure $\mu_Q^{0}$, where $\text{spt }\mu_Q^{0} \subset E \cap \bar{Q}$, $\mu^{0}_{Q}(B(x, r)) \leq C r^{\beta}$ for all $x \in \RR$, $r > 0$, and $C$ does not depend on $Q$, and 
\begin{equation}
\label{mu_Q0Q}
    \mu_Q^{0}(Q) \geq \mathcal{H}^{d}_{\infty}(E\cap Q) \geq 2^{-1} l(Q)^{\beta}.
\end{equation}

To construct the measure $\mu$ supported on $E$, we first normalize each $\mu_{Q}^{0}$ according to $\varphi$. Let 
\begin{equation*}
    \mu_Q = \frac{w(Q)}{\mu_Q^{0}(Q)} \mu_Q^{0}, \text{ where }  w(Q) = \int_{Q} \varphi(x) dx.
\end{equation*}

Then by (\ref{varphi_choice}),
\begin{equation}
\label{small_q_total}
    \mu_Q(Q) = w(Q) \leq \norm{\varphi}_{L^{\infty}} |T_{\mathbf{Q}}(Q)| = 2^{1-T} = 2 l(Q),
\end{equation}

In addition, by (\ref{mu_Q0Q}),
\begin{equation}
\label{small_q_f}
    \mu_Q (B(x, r)) \leq \frac{2 l(Q)}{ 2^{-1} l(Q)^{\beta}}C r^{\beta} \leq 2^{2+{T(\beta-1)}} C r^{\beta} \leq 4C r^{\beta}.
\end{equation}

Let $\mu = \sum_{Q \in \text{ch}(\mathbf{Q})} \mu_Q$, then $\text{spt }\mu \subset E$. We claim that $\mu$ is the desired measure. First, the total mass is
\begin{equation}
\label{big_q_total}
    \mu(\mathbf{Q}) = \sum_{Q \in \text{ch}(\mathbf{Q})} w(Q) = \sum_{Q \in \text{ch}(\mathbf{Q})} \int_{Q} \varphi(x)dx = 1. 
\end{equation}

Second, $\mu$ is a $\beta$-Frostman's measure.
\begin{itemize}
    \item If $r \leq 2^{-T}$, $B(x, r)$ intersects at most four $Q \in \text{ch}(\mathbf{Q})$, so $\mu(B(x, r)) \leq 16 Cr^{\beta}$ from (\ref{small_q_f}).
    \item If $r \geq 1$, $\mu(B(x, r)) \leq 1 \leq r^{\beta}$ from (\ref{big_q_total}). 
    \item If $2^{-T} \leq r \leq 1$, by (\ref{small_q_total}),
    \begin{equation*}
    \begin{split}
        \mu(B(x, r)) \leq & \sum_{Q \in \text{ch}(\mathbf{Q}), Q \bigcap B(x, r) \not = \emptyset} \mu_Q(Q) \\
        \leq & 8\left(\frac{r}{2^{-T}} \right) 2^{1-T}  =   16r  \leq 16r^{\beta}.
    \end{split}
    \end{equation*}
\end{itemize}
    
Third, $\mu$ satisfies the spectral gap condition.
\begin{itemize}
    \item For each $Q \in \text{ch}(\mathbf{Q})$, 
    $$\mu(Q) = w(Q) = \int_{Q}\varphi(x) dx.$$
    Therefore, for $c_Q$, the center of the cube $Q$,
    $$\int_{Q} e^{-2\pi i c_Q \xi} d\mu(x) = \int_{Q} e^{-2\pi i c_Q \xi} \varphi(x) dx.$$

    For a $\xi \in \RR$, the map $x \to e^{-2\pi i x \xi}$ is $D |\xi|$-Lipschitz for a $D>0$. Since $|x-c_Q| \leq 2^{-T}$ if $x \in Q$,
    \begin{equation*}
        \begin{split}
            |\widehat{\mu}(\xi) - \widehat{\varphi}(\xi)| 
            = & \left|  \int e^{-2\pi i x \xi} d\mu(x) - \int e^{-2\pi i x \xi}  \varphi(x) dx\right| \\
            \leq & \sum_{Q \in \text{ch}(\mathbf{Q})}  \int_{Q} |e^{-2\pi i x \xi} - e^{-2\pi i c_Q \xi} | d\mu(x) +  \int_{Q} |e^{-2\pi i x \xi} - e^{-2\pi i c_Q \xi}|  \varphi(x) dx \\
            \leq & 2^{-T+1}D |\xi|.
        \end{split}
    \end{equation*}

    Therefore,
    \begin{equation*}
    \begin{split}
         \int_{A^{\frac{1}{5}} \leq |\xi| \leq B^2} |\widehat{\mu}(\xi)| d\xi \leq &   \int_{A^{\frac{1}{5}} \leq |\xi| \leq B^2} |\widehat{\mu}(\xi) - \widehat{\varphi}(\xi)| d\xi +  \int_{A^{\frac{1}{5}} \leq |\xi| \leq B^2} | \widehat{\varphi}(\xi)| d\xi \\
         \leq & D 2^{-T+1} \int_{ |\xi| \leq B^2}  |\xi| d\xi +  \int_{|\xi| \geq A^{\frac{1}{5}}} | \widehat{\varphi}(\xi)| d\xi \\
         \leq & D 2^{-T+1} B^4 +  2^{-1}A^{-3} \leq A^{-3}
    \end{split}
    \end{equation*}
    by our choice of $T$.
\end{itemize}

Finally, we apply the following lemma to obtain a bound on the $s$-energy integral of $\mu$.
\end{proof}
\begin{lemma}\cite[Section 2.5]{mattila_2015}
\label{lem_fn}
    Suppose that $\nu \in \mathcal{M}([0, 1])$, and there exists $C>0$ such that $\nu(B(x, r)) \leq C r^{\beta}$ for any $x \in [0, 1], r > 0$. Then, for $s< \beta$, $I_{s}(\nu) \leq 1+\frac{Cs}{\beta-s}$.
\end{lemma}

\subsection{Sobolev improving estimate}
\label{sec_sie}

Proposition \ref{cor_sie} is deduced from the following proposition.

\begin{ppt}\cite{fgp}
\label{th_sie}
There exist $\gamma_0>0$ and $\kappa >0$ such that the following statement holds.
Let $\gamma \in (0, \gamma_0)$, $p_0 > 0$, and $q_0 \geq 0$. Then, there exists $C_{\gamma, p_0, q_0} > 0$ such that
\begin{equation*}
\left|\int \int f(x+t) g(x+P_{p, q}(t)) h(x) \tau_l(t)dt  dx \right| \leq C_{\gamma, p_0, q_0} 2^{\kappa l} \norm{f}_{H^{-\gamma}}  \norm{g}_{H^{-\gamma}} \norm{h}_{H^{-\gamma}}
\end{equation*}
for all $f, g, h \in \mathcal{S}(\RR)$ if $P_{p, q}(t) = pt^2+qt$ for $|p| \geq p_0$ and $|q| \leq q_0$.
\end{ppt}
 We refer readers to the proof in section 2 of \cite{fgp} that yields Proposition \ref{th_sie}, which is stronger than the version in that paper for quadratics of the form $P_{p, q}(t)=pt^2+qt$. The proof starts by applying (\ref{tof}) and Littlewood-Paley decomposition. Then, the methods of stationary phase and estimates in \cite{li13} and \cite{hor} are used.

\begin{proof}[Proof of Proposition \ref{cor_sie} given Proposition \ref{th_sie}]
    The function $T: \RR \to \CC$ given by $$T(x) = \int f(x+t) g(x+P_{p, q}(t))\tau_l(t)dt$$ is continuous. Since $\mu_{\eps} \to \mu$ weakly as $\eps \to 0$,
    \begin{align*}
           & \left|\int \int f(x+t) g(x+P_{p, q}(t))\tau_l(t)dt d\mu(x) \right| \\
           = & \lim_{\eps \to 0 } \left|\int T(x) \mu_{\eps}(x) dx \right| \\
           \leq & \lim_{\eps \to 0 }  C_{\gamma, p_0, q_0}  2^{\kappa l}\norm{f}_{H^{-\gamma}} \norm{g}_{H^{-\gamma}}  \norm{\mu_{\eps}}_{H^{-\gamma}}. \qedhere
    \end{align*}
\end{proof}

\bibliographystyle{plain} 
\bibliography{refs}

\end{document}